\documentclass[leqno,10pt]{amsart}
\usepackage{amsthm,amssymb,amsmath,amscd,latexsym}

\usepackage[latin1]{inputenc}
\usepackage[all]{xy}
\usepackage{graphicx}

\theoremstyle{plain}
\newtheorem{theorem}{Theorem}[section]
\newtheorem{proposition}[theorem]{Proposition}
\newtheorem{lemma}[theorem]{Lemma}
\newtheorem{corollary}[theorem]{Corollary}

\theoremstyle{definition}

\newtheorem{definition}[theorem]{Definition}
\newtheorem{example}[theorem]{Example}

\theoremstyle{remark}

\numberwithin{equation}{section} \makeatletter
\let\c@equation\c@theorem
\makeatother

\def\paritem#1{%
  \smallskip
  \setbox0=\hbox{#1\enspace}
  \par\noindent
  \ifnum\wd0>\parindent\box0
  \else\hbox to\parindent{\box0\hfill}\fi\ignorespaces}

\DeclareMathOperator\St{St}
\DeclareMathOperator\st{st}

\def\epsilon{\varepsilon}

\begin{document}
\title[Lattice Gauge Field Theory \ldots] {Lattice Gauge Field
Theory and Prismatic Sets}
\author[B.Akyar ]{B.Akyar}
\address{Department of Mathematics   \\
Dokuz Eyl\"ul University \\
TR-35160 Izmir, Turkey} \email[B.~Akyar]{bedia.akyar@deu.edu.tr}
\author[J.~L.~Dupont]{J. ~L. ~Dupont}
\address{Department of Mathematics \\
University of Aarhus \\
DK-8000 {\AA}rhus C, Denmark} \email[J.~L.~Dupont]{dupont@imf.au.dk}

\date{\today}

\begin{abstract}
We study prismatic sets analogously to simplicial sets except that realization
involves prisms, i.e., products of simplices rather than just simplices.
Particular examples are the prismatic subdivision of a simplicial set $S$ and
the prismatic star of $S$. Both have the same homotopy type as $S$ and in
particular the latter we use to study lattice gauge theory in the sense of Phillips
and Stone. Thus for a Lie group $G$ and a set of parallel transport functions defining
the transition over faces of the simplices, we define a classifying map from the
 prismatic star to a prismatic version of the classifying space of $G$. In turn this
defines a $G$-bundle over the prismatic star.

\end{abstract}

\maketitle

\tableofcontents


\medbreak
\section{\bf{Introduction}\label{one}}

In the study of global properties of locally trivial fibre bundles
it is a fundamental difficulty that the usual combinatorial methods
of algebraic topology depends on the use of simplicial complexes
which structure behaves badly with respect to local trivializations.
By a theorem of Johnson \cite{J}, the base and total space of a
locally trivial smooth fibre bundle with projection $\pi:E\to{B}$
can be triangulated in such a way that $\pi$ is a simplicial map.
But obviously even in this case a general fibre is not a simplicial
complex in any natural way. However such a fibre has a natural
decomposition into prisms, i.e., products of simplices, and the
whole triangulated bundle gives the basic example of a prismatic
set, analogous to the notion of a simplicial set derived from a
simplicial complex. Prismatic sets were introduced and used by the
second author and R. Ljungmann in \cite{DLj} (see also Ljungmann's thesis \cite{Lj})
in order to construct an explicit fibre integration map in smooth
Deligne cohomology, see also \cite{DK}. But the important special case of the prismatic
subdivision of a simplicial set was used in Akyar \cite{B} in connection
with ``Lattice Gauge Theory'' in the sense of Phillips and Stone
\cite{PS}, \cite{PS1}, \cite{PS5} and similar constructions have been used in other
connections, see e.g. \cite{MS}. One can see \cite{L} for further information about Lattice Gauge Fields.

In this paper we shall give a more systematic treatment of prismatic
sets and their properties but we shall concentrate on the
applications to lattice gauge theory extending the work of Phillips
and Stone to arbitrary simplicial sets and all dimensions. For an 
arbitrary simplicial set $S$ and a given
Lie group $G$ together with a set of parallel transport functions in their
sense, we construct a prismatic
set $\bar{P}(S)$ of the same homotopy type as $S$ and a classifying map
from $\bar{P}(S)$ to a prismatic version of the standard model for
$BG$. This is one of our main results (Theorem 8.1). Geometrically, for $S$ a simplicial complex,
 $\bar{P}(S)$ is closely related to the nerve of the covering by stars of vertices (Theorem 5.1). In turn this gives a principal
$G$-bundle with a connection and thus in principle gives rise via the usual Chern-Weil and Chern-Simons theory
to explicit formulas for characteristic classes (Corollary 8.2). We shall return to this elsewhere. One can see \cite{CrS}, \cite{CnS}, \cite{D},
\cite{F3}, \cite{W} for further information about Chern-Simons Theory.

The paper is organized as follows:

In chapter 2, prismatic sets are defined and 
their various geometric realizations are studied.

The third chapter introduces the prismatic triangulation of a simplicial
map and in particular of a simplicial set. Furthermore,
we comment on the calculation of the homology of the geometric
realization of a prismatic set.

In chapter 4 we study prismatic sets associated to stars of simplicial complexes.
It turns out that the prismatic set $\bar{P}_.(S)$ given in this
chapter in the case of a simplicial complex is the nerve of the covering by stars of vertices. 

In the fifth chapter, we compare the two star simplicial sets and prove
that there is a natural surjective map
$\bar{p}:\bar{P}_.(S)\to{P\St{S}}$. It turns out that this map is an
isomorphism for $S=K^s$, where $K$ is a simplicial complex.

In chapter 6, we introduce a prismatic version of the classifying space. This is done by replacing the Lie group $G$ by the
singular simplicial set of continuous maps $\text{Map}(\Delta^q,G)$.

In chapter 7, we introduce the notion of  ``compatible transition functions'' similar to the 
``parallel transport functions'' of Phillips-Stone \cite{PS1} for a simplicial complex $K$. 
We show how a given bundle on the realization of a simplicial set and socalled ``admissible trivializations''
give rise to a set of compatible transition functions and vice versa. We end the chapter with a remark on the relation between 
the compatible transition functions and parallel transport
along a piecewise linear path.

Finally in the last chapter we construct the classifying map for a given set of compatible transition functions. For this we
construct a prismatic map from $\bar{P}(S)$ to the prismatic model for the classifying space constructed in chapter 6.

\textbf{Acknowledgements:} We would like to thank Marcel B\"{o}kstedt for his interests and comments during the preparation of this paper.

 \newpage

\medbreak
\section{\bf{Prismatic Sets}\label{two}}

Prismatic sets are similar to simplicial sets but they are realized by using prisms instead of only simplices.

Let $\Delta^p=\{(t_0,\dots,t_p)\in{{\Bbb{R}}^{p+1}}\text{ }|\text{
}\sum_i t_i=1,t_i\leq{1}\}$ be a standard $p$-simplex given with
barycentric coordinates. A prism is a product of simplices, that is,
a set of the form $\Delta^{q_0...q_p}=\Delta^{q_0}\times{\dots}\times{\Delta^{q_p}}$.

The motivating example is  triangulated fibre bundles:

\begin{example}
\label{2.1}

Given a smooth fibre bundle $\pi:Y\to{Z}$ with
dim$Y=m+n$, dim$Z=m$ and compact fibres possibly with boundary. By
a theorem of Johnson \cite{J}, there are smooth triangulations $K$ and $L$
of $Y$ and $Z$, respectively and a simplicial map $\pi
^\prime:K\to{L}$ in the following commutative diagram

\begin{displaymath}
\xymatrix{ |K| \ar[d]^{|\pi^{\prime}|} \ar[r]^{\approx}
& Y  \ar[d]^\pi \\
|L| \ar[r]^{\approx} & Z }
\end{displaymath}
and the horizontal maps are homeomorphisms which are smooth on each
simplex, here $|K|=\bigcup_{\tau\in{K_k}}\Delta^k\times{\tau}/\sim$,
$k=0,...,\text{dim}K$, is the geometric realization. \vspace{0.5 cm}

One can extend a given such  triangulation of $\partial Y\to{Z}$
to a triangulation of $Y\to{Z}$.

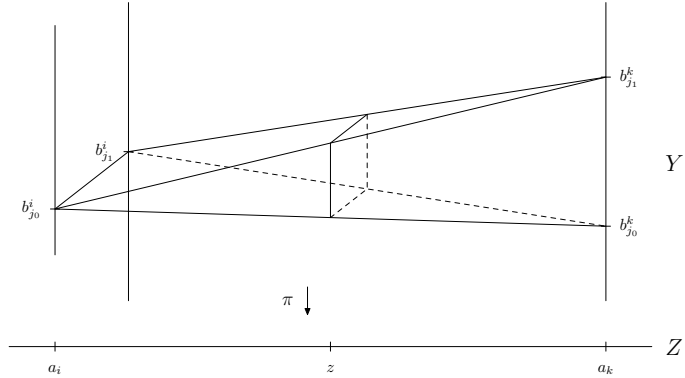
\begin{figure}[htbp]
  \includegraphics[width=9cm]{figs.1}\\
\end{figure}

A simplex $\tau$ in $K$ has vertices
$\tau=(b_0^0,\dots,b_{q_0}^0|...|b_0^p,\dots,b_{q_p}^p)$ with
$\sigma=(a_0,\dots,a_p)$ such that $\pi^\prime(b_j^i)=a_i$. Here, we give the set of vertices
of the total space, the lexicographical order. So geometrically, for an open simplex $\stackrel{\circ}{\sigma}$ in $L$, we
have
\begin{displaymath}
\pi^{-1}(\lvert\stackrel{\circ}{\sigma}\rvert)\approx{\lvert\stackrel{\circ}{\sigma}\rvert}\times{\bigcup_{\tau\in{\pi^{-1}(\sigma)}}}\Delta^{q_0...q_p}\times{\tau}.
\end{displaymath}

\end{example}

We collect all these in the formal definition below using simplicial sets. For these we recall the notation but refer otherwise to Mac Lane \cite{McL}, May \cite{May}.

\begin{definition}
  \label{2.2}
A simplicial set $S_.=\{S_q\}$ is a sequence of sets with face
operators $d_i:S_q\to{S_{q-1}}$ and degeneracy operators
$s_i:S_q\to{S_{q+1}}$, $i=0,...,q$, satisfying the following
identities:
\begin{eqnarray*}
d_id_j= \left\{ \begin{array}{r@{\quad: \quad}l}d_{j-1}d_i & i<j  \\
d_jd_{i+1} & i\geq{j}, \end{array} \right.
\end{eqnarray*}
\begin{eqnarray*}
s_is_j= \left\{ \begin{array}{r@{\quad: \quad}l}s_{j+1} s_i & i\leq{j}  \\
s_j s_{i-1} & i>j, \end{array} \right.
\end{eqnarray*}
and
\begin{displaymath}
d_is_j= \left\{ \begin{array}{r@{\quad:
\quad}l}s_{j-1}d_i & i<j \\ \text{id} & i=j,i=j+1
\\ s_jd_{i-1} & i>j+1.  \end{array} \right.
\end{displaymath}
\end{definition}

\begin{example}
\label{2.3}
A simplicial complex $K_.$ gives a simplicial set where
\linebreak
$K_p=\{(a_{i_0},...,a_{i_p})\text{   }|\text{
}\text{some nondecreasing sequences for a given linear ordering of}\text{   }
K_0\}$ is the set of $p$-simplices.

\end{example}

\begin{example}
\label{2.4}
 Given an open cover $\mathcal{U}={U_i}$ of $Z$ we have the nerve \linebreak
$N\mathcal{U}=\{N\mathcal{U}(p)\}$ of the covering, where
\begin{displaymath}
N\mathcal{U}(p)=\bigsqcup_{i_0,...,i_p}U_{i_0}\cap{...}\cap{U_{i_p}},
\end{displaymath}
and $(i_0,...,i_p)$ is nondecreasing for a given linear order of the index set.

Let us denote $U_{i_0}\cap{...}\cap{U_{i_p}}$ by $U_{i_0,...,i_p}$. $N\mathcal{U}$ is a simplicial manifold, where the face and degeneracy maps come from the followings

\begin{eqnarray*}
d_j:U_{i_0,...,i_p}&\to&{U_{i_0,...,{\hat{i}}_j,...,i_p}}\\
s_j:U_{i_0,...,i_p}&\to&{U_{i_0,...,i_j,i_j,...,i_p}}
\end{eqnarray*}
That is, $N\mathcal{U}(p)$ is a smooth manifold for each $p$ and the face and degeneracy maps are smooth.
There is also a corresponding simplicial set $N_d\mathcal{U}=\{N_d\mathcal{U}(p)\}$ called the discrete nerve of the covering. Here
$N_d\mathcal{U}(p)$ is simply the set consisting of an element for each non-empty intersection of $p+1$ open sets from $\mathcal{U}$. So there is
a natural forgetful map $N\mathcal{U}\to{N_d\mathcal{U}}$.

\end{example}
{\bf Note:} If $S_.$ has only face operators, then it is called a
$\Delta$-set.

\begin{definition}
  \label{2.5}

Given $p\geq{0}$, a $(p+1)$-multi-simplicial set is a sequence
$\{S_{q_0,...,q_p}\}$ which is a simplicial set in each variable
$q_i$, $i=0,...,p$.
\end{definition}

\begin{definition}
  \label{2.6}
A prismatic set $P=\{P_{p,.}\}$ is a sequence
$P_{p,.}=\{P_{p,q_0,...,q_p}\}$ of $(p+1)$-multi-simplicial sets,
i.e., with face and degeneracy operators

\begin{eqnarray*}
d_j^i:P_{p,q_0,...,q_p}&\to&{P_{p,q_0,...,q_i-1,...,q_p}}\\
s_j^i:P_{p,q_0,...,q_p}&\to&{P_{p,q_0,...,q_i+1,...,q_p}}
\end{eqnarray*}
such that $d_j^i$, $s_j^i$ commute with $d_l^k$, $s_l^k$ for
$i\neq{k}$, and such that  $d_j^i$, $s_j^i$ for fixed $i$ satisfy the identities in Definition \ref{2.2}.

Furthermore there are face operators
\begin{displaymath}
d_k:P_{p,q_0,...,q_p}\to{P_{p-1,q_0,...,\hat{q}_k,...,q_p}}
\end{displaymath}
commuting with $d_j^i$ and $s_j^i$ (interpreting  $d_j^k=s_j^k=\text{id}$ on the right)
such that $\{P_{p,.}\}$ is a $\Delta$-set.

\end{definition}

\begin{definition}
  \label{2.7}
If similarly $P$ has degeneracy operators
\begin{displaymath}
s_k:P_{p,q_0,...,q_p}\to{P_{p+1,q_0,...,q_k,q_k,...,q_p}}
\end{displaymath}
then $P$ is called a strong prismatic set.
\end{definition}
{\bf Remark:} In this case $(P_p,d_k,s_k)$ is a usual simplicial set.

In general, degeneracy operators given in Definition \ref{2.7} do not exist naturally so in this case $(P_{p,.},d_k)$ is only a $\Delta$-set.

Example \ref{2.1} continued: A triangulated fibre bundle
\begin{displaymath}
\pi:|K|\to{|L|}
\end{displaymath}
gives a strong prismatic set $P_.(K/L)$
by letting
\begin{displaymath}
P_p(K/L)_{q_0...q_p}\subseteq{K_{p+q_0+...+q_p}\times{L_p}}
\end{displaymath}
be the subset of pairs of simplices $(\tau,\sigma)$ so that $q_i+1$
of the vertices in $\tau$ lies over the $i$-th vertex in $\sigma$.
Then we have face and degeneracy operators defined in the obvious
way. It is now straight forward to check that this is a strong
prismatic set.

\begin{example}
\label{ES}
For a given simplicial set $S$, consider
$E_pS=\underbrace{S_.\times{...}\times{S_.}}_{p+1-times}$.\linebreak
$\pi_i:E_pS\to{E_{p-1}S}$ is the projection which deletes the
$i$-th factor. Similarly, the diagonal map
$\delta_i:E_pS\to{E_{p+1}}S$ repeats the $i$-th factor. This is a strong prismatic set.

\end{example}

Prismatic sets have various geometric realizations.

\begin{definition}
  \label{2.9}
First, we have for each $p$ the thin (geometric) realization
\begin{eqnarray}
\label{2.10}
|P_{p,.}|=\bigsqcup_{q_0,...,q_p}\Delta^{q_0...q_p}\times{P_{p,q_0,...,q_p}}/\sim
\end{eqnarray}
with equivalence relation ``$\sim$'' generated by the face and degeneracy maps
\begin{eqnarray*}
\varepsilon_j^i&:&\Delta^{q_0...q_i...q_p}\to{\Delta^{q_0...q_i+1...q_p}}\text{   }\text{   }\text{and}\\
\eta_j^i&:&\Delta^{q_0...q_i...q_p}\to{\Delta^{q_0...q_i-1...q_p}},
\end{eqnarray*}
respectively.
$\{|P_{p,.}|\}$ is a $\Delta$-space hence it
gives a fat realization
\begin{eqnarray}
\label{2.10*} \|\text{   }|P_.|\text{
}\|=\bigsqcup_{p\geq{0}}\Delta^p\times{|P_{p,.}|}/\sim
\end{eqnarray}
by only using face operators $d_k$.

The face and degeneracy operators
$d_k$, $s_k$ act on $\Delta^{q_0...q_p}$ as the projection and the
diagonal, respectively so they induce a structure of a simplicial
set on $|P_p|$. In other words, the projection
$\pi_i:\Delta^{q_0...q_p}\to{\Delta^{q_0...{\hat{q}}_i...q_p}}$
deletes the $i$-th coordinate and the diagonal map
$\Delta_i:\Delta^{q_0...q_p}\to{\Delta^{q_0...q_iq_i...q_p}}$
repeats the $i$-th factor. Then the further equivalence relation on $|P_.|$ given in (\ref{2.10*}) is generated by
\begin{eqnarray*}
(\varepsilon^it,s,\sigma)\sim(t,\pi_is,d_i\sigma),\text{   }t\in{\Delta^{p-1}},\text{   }s\in{\Delta^{q_0...q_p}},\text{   }\sigma\in{P_{p,q_0,...,q_p}}.\\
\end{eqnarray*}

If $P_.$ is strong then we also have a thin realization
\begin{displaymath}
|P_.|=\||P_.|\|/\sim
\end{displaymath}
given by the above and the further relation
\begin{eqnarray*}
(\eta^it,s,\sigma)\sim(t,\Delta_is,s_i\sigma),\text{
}t\in{\Delta^{p+1}},\text{   }s\in{\Delta^{q_0,...,q_p}},\text{
}\sigma\in{P_{p,q_0,...,q_p}}.
\end{eqnarray*}

Similarly, we can define for each $p$, the fat
realization $||P_.||$, that is
\begin{displaymath}
\|P_p\|=\bigsqcup_{q_0,...,q_p}\Delta^{q_0...q_p}\times{P_{p,q_0,...,q_p}}/\sim
\end{displaymath}
with equivalence relation given by only the face maps
$d_j^i$.

Moreover, we have the very fat realization
\begin{displaymath}
\|\text{   }\|P_.\|\text{
}\|=\bigsqcup_{p\geq{0}}\Delta^p\times{\|P_{p,.}\|}/\sim
\end{displaymath}
using only face operators.

For a given simplicial set $S$ and $E_pS$  as in Example \ref{ES} we have
$\|\text{   }|E_.S|\text{   }\|$ as the fat realization of the space which maps
 $p$-th term to $\underbrace{|S_.|\times{...}\times{|S_.|}}_{p+1-times}$.

Define for a space $X$, $E_pX=\underbrace{X_.\times{...}\times{X_.}}_{p+1-times}$
Let us say $X=|S|$ then $\|\text{   }|E_.S|\text{   }\|$ is contractible.

\end{definition}

\newpage

\medbreak
\section{\bf{Prismatic Triangulation}\label{three}}

Let us return to the case of a triangulated fibre bundle $|K|\to{|L|}$.
In this case the natural map
\begin{displaymath}
P_p(K/L)_{q_0,...q_p}\to{K_{q_0+...+q_p+p}}
\end{displaymath}
induces a homeomorphism

\begin{displaymath}
\xymatrix{ |P_.(K/L)| \ar[d]^{|\pi^{\prime}|} \ar[r]^{\phantom{123}\approx}
&|K|  \ar[d]^\pi \\
|L| \ar[r]^{=} & |L| }
\end{displaymath}
In this diagram, the top horizontal map we shall call the prismatic triangulation
homeomorphism
\begin{displaymath}
\lambda:|P_.(K/L)|\stackrel{\cong}\to{|K|}
\end{displaymath}
induced by
\begin{eqnarray}
\label{3.1}
\lambda(t,s^0,...,s^p,(\tau,\sigma))=(t_0s^0,...,t_ps^p,\tau)\in{\Delta^{p+q}\times{K_{p+q}}},
\end{eqnarray}
 where $(t,s,\tau,\sigma)\in{\Delta^p}\times{\Delta^{q_0...q_p}}\times{P_p(K/L)_{q_0...q_p}}$ and $q=q_0+...+q_p$.

{\bf Note:} If $\stackrel{\circ}{\sigma}$ is an open $p$-simplex in
$L$ then $\lambda$ provides a natural trivialization of
$|K|_\sigma=\pi^{-1}(\stackrel{\circ}{\sigma})$, that is, a
homeomorphism
\begin{displaymath}
\lambda:\stackrel{\circ}{\sigma}\times{|P_p(K/\sigma)}|\stackrel{\approx}\to{|K|_\sigma}.
\end{displaymath}

We can generalize this construction to any simplicial map:

\begin{example}
\label{3.3}
{\it Prismatic triangulation of a simplicial map.}
Let $f:S_.\to{\bar{S_.}}$ be a simplicial
map of simplicial sets and define $P_.(f)$
by
\begin{displaymath}
P_p(f)_{q_0,...,q_p}=\{(\sigma,\bar{\sigma})\in{S_{q_0+...+q_p+p}}\times{{\bar{S}}_p}\text{
  }|\text{   }
f(\sigma)=\mu_{q_0,...,q_p}(\bar{\sigma})\}
\end{displaymath}
where the corresponding map
\begin{displaymath}
 \mu^{q_0,...,q_p}:\Delta^{q_0+...+q_p+p}\to{\Delta^p}
\end{displaymath}
 is given by
\begin{displaymath}
\{0,...,q_0|...|q_0+...+q_{p-1}+p,...,q_0+...+q_p+p\}\to{\{0,...,p\}}.
\end{displaymath}
By this, we mean that the basis vectors $e_0,...,e_{q_0}$ are mapped to
$e_0$, and $e_{q_0+1},...,e_{q_0+q_1+1}$ are mapped to $e_1$ and etc.
Explicitly
\begin{displaymath}
\mu_{q_0,...,q_p}=\hat{s}_{q+p}\circ{s_{(q_0+...+q_p+p-1)...(q_0+...+q_{p-1}+p)}}\circ{...}\circ{\hat{s}_{q_0}}\circ{s_{(q_0-1)...(0)}},
\end{displaymath}
where the $\hat{s}_i$ are left out and
\begin{displaymath}
s_{(q_0+...+q_i+i-1)...(q_0+...+q_{i-1}+i)}=s_{q_0+...+q_i+i-1}\circ{...}\circ{s_{q_0+...+q_{i-1}+i}},
\end{displaymath}
$i=0,...,p$. The boundary maps in the fibre direction
\begin{displaymath}
d_j^i:P_p(f)_{q_0,...,q_p}\to{P_p(f)_{q_0,...,q_i-1,...,q_p}}
\end{displaymath}
are inherited from the face operators defined on $S_{q+p}$. Thus
\begin{displaymath}
d_j^i(\sigma,\bar{\sigma})=(d_{q_0+...+q_{i-1}+i+j-1}\sigma,\bar{\sigma}).
\end{displaymath}
Similarly the degeneracy maps $s^i_j$ on $P_p(f)_{q_0,...,q_p}$
\begin{displaymath}
s^i_j:P_p(f)_{q_0,...,q_p}\to{P_p(f)_{q_0,...,q_i+1,...,q_p}}
\end{displaymath}
are inherited from the ones on $S_{q+p}$. That is,
\begin{displaymath}
s_j^i(\sigma,\bar{\sigma})=(s_{q_0+...+q_{i-1}+i+j-1}\sigma,\bar{\sigma}).
\end{displaymath}
The boundary maps
\begin{displaymath}
d^i:P_p(f)_{q_0,...,q_p}\to{P_{p-1}(f)_{q_0,...,{\hat{q}}_i,...,q_p}}
\end{displaymath}
are determined by the boundary maps defined on both $S_{q+p}$ and
${\bar{S}}_p$. Thus
\begin{displaymath}
d^i(\sigma,\bar{\sigma})=(d_{q_0+...+q_{i-1}+i-1}\circ{...}\circ{d_{q_0+...+q_i+i-1}}\sigma,d_i\bar{\sigma}),
\end{displaymath}
here the composition of the face operators can be shortly written as
\begin{displaymath}
d_{(q_0+...+q_{i-1}+i-1)...({q_0+...+q_i+i-1})}=d_{q_0+...+q_{i-1}+i-1}\circ{...}\circ{d_{q_0+...+q_i+i-1}}.
\end{displaymath}
\end{example}

{\bf Note:} $P_.(f)$ is a prismatic set, but in general not a strong one.

\begin{theorem}
  \label{3.4}
There is a pullback diagram

\begin{displaymath}
\xymatrix{ \|\text{   }|P_.(f)|\text{   }\| \ar[d]^{\|f\|}
\ar[r]^{\phantom{12}\lambda}
&|S_.|  \ar[d]^{|f|} \\
\|\bar{S}_.\| \ar[r]^q & |\bar{S}_.| }
\end{displaymath}

In particular $\lambda$ is a homotopy equivalence.

\end{theorem}

{\bf Proof:} The map
$\lambda:\Delta^p\times{\Delta^{q_0...q_p}}\times{P_p(f)_{q_0...q_p}}\to{\Delta^{q+p}\times{S_{q+p}}}$
is given by
$\lambda(t,s,\sigma,\bar{\sigma})=(t_0s^0,...,t_ps^p,\sigma).$ The
commutativity of the diagram follows from the definition of $P_.(f)$
since
\begin{displaymath}
P_p(f)_{q_0,...,q_p}\subseteq{S_{q+p}\times{{\bar{S}}_p}}
\end{displaymath}
consists of pairs
$(\sigma,\bar{\sigma})\in{S_{q+p}}\times{{\bar{S}}_p}$ such that
$f(\sigma)=\mu_{q_0,...,q_p}(\bar{\sigma})\in{{\bar{S}}_p}$.

By the commutativity of the diagram,
$\lambda$ factors over the pullback
$|S_.|\times_{|\bar{S}_.|}{\|\bar{S}_.\|}$ in the diagram

\begin{displaymath}
\xymatrix{ |S_.|\times_{|\bar{S}_.|}{\|\bar{S}_.\|} \ar[d]^{\|f\|}
\ar[r]^{\phantom{1234}{\text{pr}}_1}
&|S_.|  \ar[d]^{|f|} \\
\|\bar{S}_.\|
\ar[r]^q & |\bar{S}_.| }
\end{displaymath}
Here elements in the pullback
$|S_.|\times_{|\bar{S}|}{\|\bar{S}_.\|}$ are represented by pairs
$((t,\sigma),(\bar{t},\bar{\sigma}))$ such that
$f(\sigma)=\mu_{q_0,...,q_p}(\bar{\sigma})$ and
$\bar{t}=\mu^{q_0,...,q_p}(t)$, where $\sigma\in{S_{q+p}}$,
$\bar{\sigma}\in{{\bar{S}}_q}$. Therefore
$\lambda\times{\|f\|}:\|\text{
  }|P_.(f)|\text{   }\|\to{|S|\times{\|\bar{S}\|}}$
induces $\Lambda$ in the diagram

\begin{displaymath}
\xymatrix{ \|\text{   }|P_.(f)|\text{   }\| \ar[r]^{\Lambda} \ar[d]
& |S_.|\times_{|\bar{S}_.|}{\|\bar{S}_.\|}
\ar[r]^{\phantom{1234}\text{pr}_1} \ar[d]^{{\text{pr}}_2}
& |S_.| \ar[d]^{|f|}\\
\|\bar{S}_.\| \ar[r]^{\text{id}} & \|\bar{S}_.\| \ar[r] &
|\bar{S}_.|}
\end{displaymath}
Now $\Lambda$ is a homotopy equivalence. Indeed, an argument similar
to the note following (\ref{3.1}) gives a homeomorphism of the
preimage $\|f\|$ of an open simplex in $\|\bar{S}_.\|$. Hence
$\Lambda$ is shown to by a homeomorphism by induction over skeleton
of $\|\bar{S}_.\|$.

$\hfill \Box$

\begin{example}
\label{3.5} {\it Prismatic triangulation of a simplicial set}. Let
$S_.$ be a simplicial set and $\bar{S_.}=*$ the simplicial set with
one element in each degree. Here \linebreak $P_p(f)=P_pS$ is called
the $p$-th prismatic subdivision of $S$ and  for each
$t\in{\stackrel{\circ}{\Delta}^p}$ the map
$\lambda_p(t,-):|P_pS|\to{|S_.|}$ is a homeomorphism. In this case,
Theorem \ref{3.4} gives a homeomorphism $\Lambda:\|\text{
  }|P_.S|\text{   }\|\stackrel{\approx}\to{\|*\|}\times{|S_.|}$, here
$\|*\|=\bigcup_n\Delta^n/\partial\Delta^n$. In particular
$\lambda:\|\text{   }|P_.S|\text{   }\|\to{|S_.|}$ is a homotopy
equivalence. We shall call $P_.S$ the prismatic triangulation of
$S$.

For later use, let us give the explicit construction of the
$p+1$-prismatic set $P_.S_.$ and its realization:

\begin{displaymath}
P_pS_{q_0,...,q_p}=S_{q_0+...+q_p+p}.
\end{displaymath}

The face operators
\begin{displaymath}
{d_j}^{(i)}:P_pS_{q_0,...,q_i,...,q_p}=S_{q+p}\to{P_pS_{q_0,...,q_i-1,...,q_p}}=S_{q+p-1}
\end{displaymath}
are defined by
\begin{displaymath}
d_j^{(i)}:=d_{q_0+...+q_{i-1}+i+j},
\end{displaymath}
$j=0,...,q_i$.
Similarly, the degeneracy operators
\begin{displaymath}
{s_j}^{(i)}:P_pS_{q_0,...,q_i,...,q_p}=S_{q+p}\to{P_pS_{q_0,...,q_i+1,...,q_i}}=S_{q+p+1}
\end{displaymath}
can be defined by
\begin{displaymath}
s_j^{(i)}:=s_{q_0+...+q_{i-1}+i+j},
\end{displaymath}
$j=0,...,q_i$.
The face maps
\begin{displaymath}
d_{(i)}:P_pS_{q_0,...,q_p}\to{P_{p-1}S_{q_0,...,{\hat{q}}_i,...q_p}}
\end{displaymath}
are the operators corresponding to
\begin{displaymath}
{\varepsilon}^{(i)}:{\Delta}^{q_0+...+{\hat{q}}_i+...+q_p+p-1}\to{{\Delta}^{q_0+...+...+q_p+p}}
\end{displaymath}
take $(e_0,...,e_{q_0+...+{\hat{q}}_i+...+q_p+p-1})$ to $(e_0,...,e_{q_0+...+q_p+p})$, deleting the elements \linebreak
$q_0+...+q_{i-1}+i,...,q_0+...+q_i+i$.
It deletes $(q_i+1)$- elements.
In contrary to this, there is no degeneracy operator.

Now we turn to the realizations. For the sequences of spaces
$\{|P_.S_.|\}$, we obtain the fat realization: \vskip 0.3 cm
\begin{displaymath}
||\text{   }|P_.S_.|\text{
}||=\bigsqcup_{p\geq{0}}{\Delta}^p\times{|P_pS_.|}/_{\sim},
\end{displaymath}
where
\begin{displaymath}
|P_pS_.|=\bigsqcup{\Delta}^{q_0...q_p}\times{S_{q_0+...+q_p+p}}/_{\sim}
\end{displaymath}
and the face operators $\pi_i:|P_pS_.|\to{|P_{p-1}S_.|}$ are given
by $\pi_i={\text{proj}}_i\times{d_{(i)}}$ with
${\text{proj}}_i:\Delta^{q_0...q_p}\to{\Delta^{q_0...{\hat{q}}_i...q_p}}$
beeing the natural projection.

Note that $\lambda_p:\Delta^p\times{|P_pS_.|}\to{|S_.|}$ satisfies
\begin{displaymath}
\lambda_p\circ{(\varepsilon^i\times{\text{id}})}=\lambda_{p-1}\circ{(\text{id}\times{\pi_i})}.
\end{displaymath}
Thus $\lambda_p$ induces the map $\lambda$ on the fat realization.

\end{example}

Let $\|\text{   }|P_.S_.|\text{   }\|^p$ respectively $\|\text{
}|S_.|\text{   }\|^p$ denote the subcomplexes generated by
\linebreak $\Delta^p\times{|P_pS|}$ respectively
$\Delta^p\times{|S_.|}$. Then the restriction of $\Lambda$ to
$\|\text{   }|P_.S_.|\text{   }\|^p$ is given by
\begin{displaymath}
\Lambda_p(t,s,\sigma)=(t,\lambda_p(t,s,\sigma)).
\end{displaymath}

\begin{corollary}
The map $\Lambda_p$ induce a homeomorphism
\begin{displaymath}
\Lambda:\|\text{   }|P_.S_.|\text{   }\|\to{\|\text{   }|S_.|\text{
}\|}\approx{\|*\|\times{|S_.|}}.
\end{displaymath}
\end{corollary}

\begin{corollary}
\label{composition} The composition map
${\text{proj}}_2\circ{\Lambda}=\lambda$
\begin{displaymath}
\|\text{   }|P_.S_.|\text{   }\|\to{\|\text{   }|S_.|\text{
}\|}\to{|S_.|}
\end{displaymath}
is a homotopy equivalence.
\end{corollary}

{\bf Remark 1:} We can calculate the homology of the geometric
realization of a prismatic set as follows:

A prismatic set $P_{.,.}$ has a double complex $(C_{p,n}(PS),\partial_F,\partial_H)$.
Here
\begin{displaymath}
C_{p,n}(PS)=\bigoplus_{q_0+...+q_p=n}C_{p,q_0,...,q_p}(PS)
\end{displaymath}
is the associated chain complex $C_p(PS)$ generated by $P_{p,q_0,...,q_p}$.
The vertical boundary map is defined by using boundary maps in the fibre direction
\begin{displaymath}
{\partial_F}^i:PC_{p,{q_0,...,q_p}}\to{PC_{p,{q_0,...,q_i-1,...,q_p}}}
\end{displaymath}
defined by ${\partial_F}^i=\sum{(-1)}^jd_j^i$, where, if $q_i=0$ then ${\partial^i}_F=0$. The total vertical boundary map is then
\begin{eqnarray*}
\label{vertical}
\partial_V={\partial^0}_F+{(-1)}^{q_0+1}{\partial^1}_F+...+{(-1)}^{q_0+...+q_{p-1}+p}{\partial^p}_F.
\end{eqnarray*}
There is also a horizontal boundary map
\begin{eqnarray*}
\label{horizontal}
\partial_H=\partial_0+{(-1)}^{q_0+1}\partial_1+...+{(-1)}^{q_0+...+q_{p-1}+p}\partial_p,
\end{eqnarray*}
where
\begin{eqnarray*}
\partial_k= \left\{ \begin{array}{r@{\quad: \quad}l}0 & \text{   }\text{if}\text{   }q_k>0  \\
d_k & \text{   }\text{if}\text{   }q_k=0, \end{array} \right.
\end{eqnarray*}
so that $\partial=\partial_V+\partial_H$ is a boundary map in the total complex $PC_*$ which is the cellular chain complex for the geometric realization.
Hence it calculates the homology. In the case of
$P_.(f)$ for $f:S\to{\bar{S}}$ a simplicial map, the double complex gives rise to a spectral sequence
which for a triangulated fibre bundle is the usual Leray-Serre spectral sequence.

{\bf Remark 2:}  For each $p$ and each
$t\in{{\stackrel{\circ}{\Delta}}^p}$,
$\lambda_p(t)^{-1}:|S|\to{\{t\}\times|P_pS|}$ induces a map of
cellular chain complexes
\begin{displaymath}
aw:C_*(S)\to{C_{*,*}(PS)}
\end{displaymath}
given by
\begin{displaymath}
aw(x)=\sum_{q_0+...+q_p=n}s_{q_0+...+q_{p-1}+p-1}\circ{...}\circ{s_{q_0}}(x)_{(q_0,...,q_p)},
\end{displaymath}
where $x\in{S_n}$.

\newpage

\medbreak
\section{\bf{Prismatic Sets and Stars of Simplicial Complexes}\label{four}}

For a simplicial set $S$ and the prismatic triangulation $P_.S$ there is another closely related prismatic set ${\bar{P}}_pS_.$ which as we
shall see for a simplicial complex is the nerve of the covering by stars of vertices considered as a prismatic set.

\begin{definition}
\label{3.7}

For $S$ a simplicial set let $\bar{P_.}S$ be the prismatic set given by
\begin{displaymath}
{\bar{P}}_pS_{q_0,...,q_p}:=S_{q_0+...+q_p+2p+1}.
\end{displaymath}
where face and degeneracy operators on $\bar{P_p}S_{q_0,...,q_p}$ are inherited
from the ones of $S_{q+2p+1}$ as follows:

Let $q=q_0+...+q_p$, the face operators
\begin{displaymath}
d_j^{(i)}:S_{q+2p+1}={\bar{P}}_pS_{q_0,...,q_p}\to{S_{q+2p}}={\bar{P}}_pS_{q_0,...,q_i-1,...,q_p}
\end{displaymath}
are defined by
\begin{displaymath}
d_j^{(i)}:=d_{q_0+...+q_{i-1}+2i+j},\text{   }j=0,...,q_i\text{   }\text{but}\text{   }j\not={2i+1+\sum_{k=0}^iq_k}.
\end{displaymath}
So $\bar{P_p}S_{q_0,...,q_p}$ has only $q+p$-face operators, i.e., we skip the following $p+1$ face operators
\begin{displaymath}
 \{d_{q_0+1},d_{q_0+q_1+3},...,d_{q+2p+1}\}.
\end{displaymath}

Similarly the degeneracy operators
\begin{displaymath}
s_j^{(i)}:S_{q+2p+1}\to{S_{q+2p+2}}
\end{displaymath}
can be defined by
\begin{displaymath}
s_j^{(i)}:=s_{q_0+...+q_{i-1}+2i+j},\text{   }j=0,...,q_i,\text{   }\text{but}\text{   }j\not={2i+1+\sum_k^iq_k}.
\end{displaymath}
Furthermore the face operators are
\begin{displaymath}
d_{(i)}:S_{q+2p+1}={\bar{P}}_pS_{q_0,...,q_p}\to{S_{q+2p-q_i-1}}={\bar{P}}_{p-1}S_{q_0,...,{\hat{q}}_i,...,q_p}
\end{displaymath} corresponding to
\begin{displaymath}
\varepsilon^{(i)}:{\Delta}^{q+2p-q_i-1}\to{{\Delta}^{q+2p+1}}
\end{displaymath}
which take $(e_0,...,e_{q_0+...+{\hat{q}}_i+...+q_p+2p-1})$ to
$(e_0,...,e_{q+2p+1})$, by deleting the vectors with
indices $(q_0+...+q_{i-1}+2i,...,q_0+...+q_i+2i+1)$. So it deletes
$q_i+2$ elements. That is,
\begin{displaymath}
d_{(i)}=d_{q_0+...+q_{i-1}+2i}\circ{...}\circ{d_{q_0+...+q_i+2i+1}},\text{   }i=0,...,p.
\end{displaymath}

\end{definition}

{\bf Remark:} As $P_.S$, ${\bar{P}}_.S$ is a prismatic set but in
general not a strong prismatic set.

{\bf Realization of ${\bar{P}_.S_.}$:}

Notice that
\begin{displaymath}
\|\text{   }|\bar{P_.}S_.|\text{
}\|=\bigsqcup_{p\geq{0}}{\Delta}^p\times{{\Delta}^{q_0...q_p}}
\times{{\bar{P}}_pS_{q_0,...,q_p}}/_{\sim}
\end{displaymath}
where the equivalence relation apart from the internal relations in $|{\bar{P}}_pS|$ using $d_j^{(i)}$ and $s_j^{(i)}$,
include the relations
\begin{displaymath}
(\varepsilon^it,(s,y))\sim(t,\pi_i(s,y)),
\end{displaymath}
with $\pi_i=({\text{proj}}_i)\times{d_{(i)}}$ the face operators on $|\bar{P_p}S_.|$.

The relation of ${\bar{P}}_.S$ with $S_.$ and $P_.S$ is as follows:

\begin{proposition}
\label{3.9}

Let $i:\|S_.\|\hookrightarrow{\|\text{   }|{\bar{P}}_pS_.|\text{   }\|}$ be an inclusion defined for \linebreak $(t,x)\in{\Delta^p\times{S_p}}$ by
\begin{displaymath}
i(t,x)=(t,1,s_0\circ{...}\circ{s_p}x)\in{\Delta^p\times{{(\Delta^0)}^{p+1}}\times{S_{2p+1}}}\subseteq{\Delta^p\times{|{\bar{P}}_pS_.|}},
\end{displaymath}
and $r:\|\text{   }|{\bar{P}}_.S_.|\text{   }\|\to{\|S_.\|}$ be the retraction defined for $(t,s,y)\in{\Delta^p\times{\Delta^{q_0...q_p}}\times{S_{q+2p+1}}}$
\begin{displaymath}
r(t,s,y)=(t,d_{0...q_0}\circ{{\hat{d}}_{q_0+1}}\circ{...}\circ{d_{(q_0+...+q_{p-1}+2p)...(q+2p)}}\circ{{\hat{d}}_{q+2p+1}}y),
\end{displaymath}
where the ${\hat{d}}_i$ are left out and $d_{(q_0+...+q_{i-1}+2i)...(q_0+...+q_i+2i)}=d_{q_0+...+q_{i-1}+2i}\circ{...}\circ{d_{q_0+...+q_i+2i}}$, $i=0,...,p$.

1) $i$ is a deformation retract with the retraction $r$.

2) There is a commutative diagram of homotopy equivalences

\begin{eqnarray*}
\xymatrix{
||S_.||\ar[r]^{i}\ar[dr] \ar[ddr]
&\|\text{   }|{\bar{P}}_.S_.|\text{   }\| \ar[d]^{f}\\
&\|\text{   }|P_.S_.|\text{   }\| \ar[d]^{\Lambda}\\
&\|\text{   }|S_.|\text{   }\|
}
\end{eqnarray*}
where $f:\Delta^p\times{\Delta^{q_0...q_p}\times{S_{q+2p+1}}}\to{\Delta^p\times{\Delta^{q_0...q_p}}}\times{S_{q+p}}$ takes $(t,s^0,...,s^p,x)$ to
\linebreak
$(t,s^0,...,s^p,d_{q_0+1}\circ{d_{q_0+q_1+3}}\circ{...}\circ{d_{q+2p+1}}x)$, $x\in{S_{q+2p+1}}$.

\end{proposition}

The proof is straight forward see \cite{B} for details.

For a simplicial complex $K$ there is another prismatic complex defined using the stars of simplices.
That is, let $K_0=\{a_i|i\in{I}\}$, where
$I=\{1,...,N\}$, be the set of vertices
and let $K_n=\{\sigma=(a_{i_0},...,a_{i_n})|i_0<...<i_n\}$ be the set of $n$-simplices such that if $\sigma\in{K_n}$ then
any face $\tau=(a_{i_{j_0}},...,a_{i_{j_k}})$ lies in $K_k$. We shall write $\tau\preccurlyeq\sigma$ in this case.
Now $K\times{K}$ is also a simplicial complex with the lexicographical order of the vertices
\begin{displaymath}
(a_i,b_j)<(a_{i^\prime},b_{j^\prime})
\Leftrightarrow\text{either}\text{   }i<{i^\prime} \text{
  }\text{or}\text{   }i=i^\prime\text{   }\text{and}\text{  }j<j^\prime,
\end{displaymath}
where $\{(a_{i_0},b_{j_0}),...,(a_{i_n},b_{j_n})\}\in{K\times{K}}$.

\begin{definition}
  \label{4.1}
  Let $K$ be a simplicial complex. The Star of $K$ is defined as
\begin{displaymath}
\text{St}(K)=\{(\sigma,\tau)\in{K\times{K}}\text{
  }|\text{   }\exists\text{   }{\sigma^\prime}\text{   }\text{such
that}\text{   }\sigma\cup\tau\preccurlyeq\sigma^\prime\}\subseteq{K\times{K}}.
\end{displaymath}
This is equivalent to say that
\begin{displaymath}
\text{St}(K)=\{\text{faces of}\text{
  }\sigma^\prime\times{\sigma^\prime}\subseteq{K\times{K}}\}.
\end{displaymath}
\end{definition}

{\bf Remark 1:} For each $\sigma\in{K}$, $(\{\sigma\}\times{K})\cap{\text{St}}(K)$ is the closure of the usual open star of $\sigma$,
i.e., the union of the open simplices having $\sigma$ as a face. Whence the name $\text{St}(K)$.
Note that $\text{St}(K)\subseteq{K\times{K}}$ is a subcomplex.

Let $K^s$ denote the simplicial set associated to the simplicial complex $K$. That is,
\begin{displaymath}
{K_n}^s=\{(a_{i_0},...,a_{i_n})\text{ }|\text{
  }\{a_{i_0},...,a_{i_n}\}\text{   }\text{a simplex of}\text{   }K\text{   }\text{(with repetitions)}\text{   }i_0\leq{...}\leq{i_n}\}.
\end{displaymath}

\newpage

$\St(K)_n ^s$ are the following:

Let $(\sigma,\tau)\in{K\times{K}}$, where
$\sigma=(a_{i_0},...,a_{i_p})$, $\tau=(b_{j_0},...,b_{j_q})$.
For $(\sigma,\tau)\in{\text{St}(K)}$, let
$\sigma^\prime=\sigma\cup\tau=(c_{k_0},...,c_{k_n})$. By allowing repetitions in Definition \ref{4.1}, i.e., by taking
$\sigma^\prime\in{K^s}$, we can assume $n=p+q$ so that either
$c_{k_s}=a_{i_t}$ or $c_{k_s}=b_{j_u}$, where $t=0,...,p$, $u=0,...,q$. Also we can assume $c_{k_n}=a_{i_p}$,
and if $a_{i_t}=b_{j_u}$ then $b_{j_u}$ comes before $a_{i_t}$.
In other words $(\sigma,\tau)$ is of the form
\begin{displaymath}
\sigma=d_{\nu_1...\nu_q}\sigma^\prime,\text{   }\tau=d_{\mu_1...\mu_p}\sigma^\prime,
\end{displaymath}
where $0\leq{\nu_1}<...<\nu_q<n$ and $0\leq{\mu_1}<...<\mu_p\leq{n}$
and $\mu_i\neq{\nu_j}$, $\forall{i,j}$. Therefore we introduce for a
general simplicial set $S$ the following.

\begin{definition}
  \label{4.2}
Let $\text{St}(S)$ be the simplicial subset of the diagonal $\delta(S\times{S})$ containing all
  simplices of the form

\begin{displaymath}
(s_{\nu_q...\nu_1}\circ{d_{\nu_1...\nu_q}}\sigma^\prime,s_{\mu_p...\mu_1}\circ{d_{\mu_1...\mu_p}}\sigma^\prime),
\end{displaymath}
where $0\leq{\nu_1}<...<\nu_q<n$ and $0\leq{\mu_1}<...<\mu_p\leq{n}$
with $\mu_i\neq{\nu_j}$, $\forall{i,j}$ as above.
Here $s_{\nu_q...\nu_1}=s_{\nu_q}\circ{...}\circ{s_{\nu_1}}$ ,
$d_{\nu_1...\nu_q}=d_{\nu_1}\circ{...}\circ{d_{\nu_q}}$, $s_{\mu_p...\mu_1}=s_{\mu_p}\circ{...}\circ{s_{\mu_1}}$ and
$d_{\mu_1...\mu_p}=d_{\mu_1}\circ{...}\circ{d_{\mu_p}}$.

\end{definition}

\begin{lemma}
  \label{4.3}
For $K$ a simplicial complex, there is a map
\begin{displaymath}
\st:\St(K)^s\to{\St(K^s)}
\end{displaymath}
which is an isomorphism.

\end{lemma}
\begin{proof}: By the discussion made before, there is a
well-defined map st. Indeed, $\text{St}(K)^s$ is a simplicial set
generated by
\begin{eqnarray*}
\{(\sigma,\tau)\in{K\times{K}}\text{ }|\text{ }\exists \text{   }
\sigma^\prime\in{K^s}\text{   }
  \text{such that}\text{   }
\sigma=d_{\nu_1...\nu_q} \sigma ^\prime,
\tau=d_{\mu_1...\mu_p} \sigma ^\prime\\
\text{  }\text{and}\text{   }
(s_{\nu_q...\nu_1}
\sigma,s_{\mu_p...\mu_1} \tau)\in{\delta(K^s\times{K^s})}
\}.
\end{eqnarray*}

By Definition ~\ref{4.2}, we can put
$\st(\sigma,\tau)=(s_{\nu_q...\nu_1}\sigma,s_{\mu_p...\mu_1}\tau)\in{\St(K^s)}$.
Clearly $\st$ is an isomorphism since for
$(\sigma,\tau)\in{\delta(K^s\times{K^s})}$ and $\sigma^\prime$ as in
Definition ~\ref{4.2}, $(\sigma,\tau)\in{K^s\times{K^s}}$ determines
an element in $K\times{K}$ by deleting repetitions and this is
unique.
\end{proof}

{\bf Remark 2:} The projection on the first factor
$\pi_1:S\times{S}\to{S}$ gives a simplicial map
$\pi_1:\text{St}_.(S)\to{S}$. Hence, we obtain a prismatic set
$P_.\text{St}(S)=P_.(\pi_1)$ as in Example \ref{3.3}. Here with
$q=q_0+...+q_p$ and
$\sigma=s_{\nu_q...\nu_1}\circ{d_{\nu_1...\nu_q}}\sigma^\prime=\mu_{q_0,...,q_p}\bar{\sigma}$,
$\tau=s_{\mu_p...\mu_1}\circ{d_{\mu_1...\mu_p}}\sigma^\prime$, we
have
\begin{displaymath}
P_p\text{St}(S)_{q_0,...,q_p}=\{(\sigma,\tau,
\bar{\sigma})\in{\text{St}(S)_{q+p}}\times{S_p}\subset{\delta}(S\times{S})_{q+p}\times{S_p}
\text{   }|\text{   } \text{ }\sigma,\tau\text{
 }\text{given above}
\}.
\end{displaymath}
That is, $\pi_1(\sigma,\tau)=\mu_{q_0,\dots,q_p}(\bar{\sigma})$, where
$\bar{\sigma}=d_{\nu_1...\nu_q}\sigma^\prime\in{S_p}$.
So The elements in $P_p\St(S)_{q_0,\dots,q_p}$ are of the form
$(\mu_{q_0,\dots,q_p}\bar{\sigma},\tau,\bar{\sigma})$, where
$\tau\in{S_q}$. Here explicitly

\begin{displaymath}
\mu_{q_0,\dots,q_p}={\hat{s}}_{q+p}\circ{s_{({q+p-1})...({q_0+...+q_{p-1}+p}})}...{\hat{s}}_{q_0+q_1+1}
s_{(q_0+q_1)...(q_0+1)}{\hat{s}}_{q_0}s_{(q_0-1)...(0)}.
\end{displaymath}

\newpage

\medbreak
\section{\bf{Comparison of the two Star Simplicial Sets}\label{five}}

We shall now prove that this is closely related to the prismatic
set $\bar{P}S$ defined in the previous section.

\begin{theorem}
    \label{4.4}

    1) There is a natural \textbf{(surjective)} map
    \begin{displaymath}
\bar{p}:\bar{P}_.S_.\to{P{\St}_.(S)_.}
\end{displaymath}

2) If $S=K^s$, where $K$ is a simplicial complex, then $\bar{p}$ is
an isomorphism.

\end{theorem}

\begin{proof}
1) Take an element
$\gamma\in{\bar{P}_pS_{q_0,...,q_p}}=S_{q_0+...+q_p+2p+1}$. Then
$\gamma$ and $q_0,...,q_p$ determine an element $\bar{p}(\gamma)$ in
$P_p\St(S)_{q_0,...,q_p}$ together with a $(p+1,q+p+1)$-partition
$(i_1,...,i_p,i_{p+1},j_1,...,j_{q+p+1})$ of $n=q+2p+1$, where
$q=q_0+...+q_p$. Here
\begin{eqnarray*}
i_1&=&q_0+1\\
i_2&=&q_0+q_1+3\\
&.&\\
&.&\\
&.&\\
i_p&=&q_0+...+q_{p-1}+2p-1\\
i_{p+1}&=&q_0+...+q_p+2p+1
\end{eqnarray*}
correspond to the $\mu_i$'s defined in Definition ~\ref{4.2} and
the $j$'s correspond to the complement, that is,
$j_1,...,j_{q_0+1},j_{q_0+2},...,j_{q_0+q_1+2},...,j_{q_0+...+q_{p-1}+p},...,j_{q_0+...+q_p+p+1}$,
are
$0,...,q_0,q_0+2,...,q_0+q_1+2,q_0+q_1+4,...,q_0+...+q_{p-2}+2p-2,...,q_0+...+q_{p-1}+2p,q_0+...+q_{p-1}+2p,...,
q_0+...+q_p+2p$, respectively. Then, in terms of Remark 2 at the end of Section 4, we define
\begin{displaymath}
\bar{p}(\gamma)=(\sigma,\tau,
\bar{\sigma})\in{P_p\text{St}(S)_{q_0...q_p}}\subseteq{S_{q+p}\times{S_{q+p}}\times{S_p}}
\end{displaymath}
where
\begin{eqnarray*}
\bar{\sigma}&=&d_{0...q_0}\circ{\hat{d}_{q_0+1}}\circ{...}
\circ{d_{(q_0+...+q_{p-1}+2p)...(q_0+...+q_p+2p)}}
\circ{\hat{d}_{q+2p+1}}(\gamma)=d_{j_1...j_{q+p+1}}(\gamma)\\
\tau&=&d_{q_0+1}\circ{d_{q_0+q_1+3}}\circ{...}
\circ{d_{q_0+...+q_p+2p+1}}(\gamma)=d_{i_1...i_{p+1}}(\gamma)\\
\sigma&=&\hat{s}_{q_0+...+q_p+p}\circ{s_{(q_0+...+q_p+p-1)...(q_0+...+q_{p-1}+p)}}\circ
\hat{s}_{q_0+...+q_{p-1}+p-1}\circ{...}\circ\\
& &\hat{s}_{q_0+q_1+1}\circ{s_{(q_0+q_1)...(q_0+1)}}\circ
\hat{s}_{q_0}\circ{s_{(q_0-1)...(0)}}(\bar{\sigma})\\
&=&s_{(q+p-1)...(q+p-q_p)}\circ{...}\circ{s_{(q_0+q_1)...(q_0+1)}}\circ{s_{(q_0-1)...(0)}}(\bar{\sigma})\\
&=&\mu_{q_0,...,q_p}(\bar{\sigma}).
\end{eqnarray*}
Using the above expression for $\bar{\sigma}$ in terms of $d$'s and $\gamma$, we
get
\begin{displaymath}
\sigma=s_{(q+p-1)...(q+p-q_p)}\circ{...}\circ{s_{(q_0+q_1)...(q_0+1)}}\circ{s_{(q_0-1)...0}}\circ{d_{j_1...j_{q+p+1}}}(\gamma).
\end{displaymath}
Now $\St(S)_{q_0+...+q_p+2p+1}$ contains the simplex
\begin{eqnarray*}
(s_{j_{q+p+1}...j_1}\circ{}s_{(q+p-1)...(q+p-q_p)}&\circ{...}&\circ{}s_{(q_0+q_1)
...(q_0+1)}
\circ{s_{(q_0-1)...0}}\circ{d_{j_1...j_{q+p+1}}}(\gamma),\\
s_{i_{p+1}...i_1}d_{i_1...i_{p+1}}(\gamma))&=&(s_{j_{q+p+1}...j_1}\sigma,s_{i_{p+1}...i_1}\tau).\\
\end{eqnarray*}

It follows that $(\sigma,\tau)\in{\St(S)}$ and hence $\bar{p}(\gamma)=(\sigma,\tau,\bar{\sigma})\in{P_p\St(S)_{q_0,...,q_p}}$.

Now $\bar{p}$ is a surjective map: Suppose
$(\sigma,\tau,\bar{\sigma})\in{P_p\St(S)_{q_0,...,q_p}}$ and we
shall find $\gamma\in{\bar{P}_p}S_{q_0,...,q_p}$ such that
$\bar{p}(\gamma)=(\sigma,\tau,\bar{\sigma})$. Thus
$(\sigma,\tau,\bar{\sigma})\in{P_p\St(S)_{q_0,...,q_p}}\subset{\delta(S\times{S})}_{q+p}$
is such that
\begin{displaymath}
\pi_1(\sigma,\tau)\in{\text{Im}\{\mu_{q_0,...,q_p}:S_p\to{S_{q+p}}\}}
\end{displaymath}
where
$\bar{\sigma}\in{S_p}$.

Again use the partition $(p+1,q+p+1)$ as above, put
$\gamma=s_{i_{p+1}...i_1}\sigma\in{S_{q+2p+1}}$. Indeed since
\begin{displaymath}
(s_{j_{q+p+1}}\circ{s_{j_q-q_p+p-1}}\circ{...}\circ{}s_{j_{q_0+q_1+2}}\circ{s_{j_{q_0+1}}}\sigma,s_{i_{p+1}...i_1}\tau)
\end{displaymath}
is of the required form as in Definition ~\ref{4.2} and since
\begin{displaymath}
(\sigma,\tau)=(d_{k_1....k_{p+1}}\times{}
d_{i_1...i_{p+1}})(s_{j_{q+p+1}}\circ{s_{j_q-q_p+p-1}}\circ{...}\circ{}s_{j_{q_0+q_1+3}}\circ{s_{j_{q_0+1}}}\sigma,s_{i_{p+1}...i_1}\tau)
\end{displaymath}
here
$d_{k_1....k_{p+1}}=d_{q_0+1}\circ{}d_{q_0+q_1+3}\circ{...}\circ{}d_{q_0+...+q_{p-1}+2p-1}\circ{}d_{q+2p+1}$.
So the $d_K$'s and the $d_I$'s are the same, where
$d_I=d_{i_1...i_{p+1}}$.

We have $(\sigma,\tau)\in{\St(S)}$. Hence
$\bar{p}(\gamma)\in{P_.\St(S)_.}$.

 2) If $S=K^s$, $K$
simplicial complex then
\begin{eqnarray*}
P_p\text{St}(K^s)_{q_0,\dots,q_p}=\{(\sigma,\tau)\in{\text{St}}(K^s)_{q+p}\subset{\delta}(K^s\times{K^s})_{q+p}\\
\text{
  }|\text{   }
\pi_1(\sigma,\tau)\in{\text{Im}}\{\mu_{q_0,\dots,q_p}:{K_p}^s\to{{K_{q+p}^s}}\}
\}.
\end{eqnarray*}

The map $\mu_{q_0,\dots,q_p}:{K_p}^s\to{{K_{q+p}^s}}$
takes $(i_0,...,i_p)$ to
$(\underbrace{i_0,\dots,i_0}_{q_0+1-\text{times}},\dots,
\underbrace{i_p,\dots,i_p}_{q_p+1-\text{times}})$. 
Then
\begin{eqnarray*}
\sigma&=&(a_{i_0},\dots,a_{i_0},\dots,a_{i_p},\dots,a_{i_p})\in{{K^s}_{q+p}},\\
\tau&=&(b_{j_0},\dots,b_{j_{q_0}},\dots,b_{j_{q_0+...+q_{p-1}+p}},\dots,b_{j_{q+p}})\in{{K^s}_{q+p}}.
\end{eqnarray*}
By the definition
$\bar{P}_p(K^s)_{q_0,\dots,q_p}=P_p(K^s)_{q_0+1,\dots,q_p+1}$. Then
$\gamma$ in ${K^s}_{q+2p+1}$ given by
$\gamma=(c_0,\dots,c_{q_0+1}|...|c_{q_0+...+q_{p-1}+2p},\dots,c_{q+2p+1})\in{K^s_{q+2p+1}}$
is uniquely determined by $\sigma$ and $\tau$.

Explicitly the inverse map
${\bar{p}}^{-1}:P_p\text{St}(K^s)_{q_0,\dots,q_p}\to{\bar{P}_p{K^s}_{q_0,\dots,q_p}}$
is defined by ${\bar{p}}^{-1}(\sigma,\tau)=\gamma$, where

\begin{eqnarray*}
\sigma&=&(a_{i_0},\dots,a_{i_p}),\\
\tau&=&(b_{j_0},\dots,b_{j_{q+p}})\text{   }\text{and}\\
\gamma&=&({j_0}^\prime,\dots,j_{q_0}^\prime,i_0|{j_{q_0+1}}^\prime,\dots,j_{q_0+q_1+1}^\prime,i_1|...|j_{q_0+...+q_{p-1}+p}^\prime,\dots,j_{q+p}^\prime,i_p)
\end{eqnarray*}
such that for
$\sum_{i=0}^{k-1}q_i+p\leq{s}\leq{\sum_{i=0}^kq_i+p}$

\begin{displaymath}
j_s^\prime= \left\{
\begin{array}{r@{\quad: \quad}l}i_{k-1} &
j_s\le{i_{k-1}} \\ j_s & i_{k-1}<j_s<i_k \\
i_k  & i_k\le{j_s},  \end{array} \right.
\end{displaymath}
$k=1,...,p$.
Hence $\gamma\in{\bar{P_p}}{K^s}_{q_0,...,q_p}$ exists and is uniquely determined by \linebreak
${(\sigma,\tau)}\in{\text{St}}(K^s)_{q+2p+1}$.

Therefore $\bar{p}:\bar{P_.}K^s\to{P_.\text{St}(K^s)}$ is an
isomorphism.

\end{proof}

{\bf Remark:} Note that $\bar{p}$ is not injective for a simplicial set in general
since for constructing the inverse map $P\text{St}(S)\to{\bar{P}S}$, there is no unique
choice for the element $\gamma$ in $\bar{P}_.S_.$. In fact, we do not know which
degeneracy operators we will use in order to define $\gamma$, so
in general the inverse is not well-defined.

\newpage

\medbreak
\section{\bf{The Classifying Space and Lattice Gauge Theory}\label{six}}

For the definition of a classifying map we need a prismatic version of the standard construction of the classifying space.

Let $G$ be a topological group and the usual classifying space $BG=EG/G$
which is constructed as a simplicial space
$EG_p=\underbrace{G_.\times{...}\times{G_.}}_{p+1-\text{times}}$,
$BG_p=(G\times{...}\times{G})/G$.

In order to make this simplicial set discrete we can replace $G$ by
the singular simplicial set of continuous maps
$S_qG=\text{Map}(\Delta^q,G)$ and $E_.S_.G$ as in Example 2.8. is a
prismatic set. However we shall need another model constructed as
follows: \vspace{0.5 cm} For a continuous map
$a\in{\text{Map}(\Delta^p\times{\Delta^{q_0...q_p}},G^{p+1})}$. Then
we define
\begin{displaymath}
a(t,s^0,...,s^p)=(a_0(t,s^0),a_1(t,s^0,s^1),...,a_p(t,s^0,...,s^p)),
\end{displaymath}
where
$(t,s^0,...,s^p)\in{\Delta^p\times{\Delta^{q_0...q_p}}}$.
$S_.G$ acts on this prismatic set and we define
\begin{eqnarray*}
P_pEG_{q_0,...,q_p}&=&\{a:\Delta^p\times{\Delta^{q_0...q_p}}\to{G^{p+1}}
\text{
  }|\text{   }
a_j(\varepsilon^it,s)\text{   }{\text{is independent of}}\\
& &\text{   } s^i\text{   }
 \text{for all}\text{   }j\text{   }\text{different from}\text{    }i
\}.
\end{eqnarray*}
$P_.BG=P_.EG/S_.G$, that is,
\begin{displaymath}
P_pBG_{q_0,...,q_p}=P_pEG_{q_0,...,q_p}/S_pG.
\end{displaymath}

\begin{proposition}
  \label{5.1}
  The evaluation maps give horizontal homotopy equivalences in the diagram

\begin{displaymath}
\xymatrix{ \|\text{   }|P_.EG_.|\text{   }\| \ar[d]^{\|\text{
}|\gamma|\text{   }\|} \ar[r]^{\phantom{123}ev}
&EG  \ar[d]^\gamma \\
\|\text{   }|P_.BG_.|\text{   }\| \ar[r]^{\phantom{123}ev} & EG/G }
\end{displaymath}
Furthermore the top map is equivariant with respect to the
homomorphism \linebreak
$ev:|S_.G|\to{G}$.
\end{proposition}

\begin{proof}
First notice that the evaluation map $ev:|S_.G|\to{G}$ is a
homotopy equivalence. Also the equivariance is obvious by the
commutative diagram
\begin{displaymath}
\xymatrix{\|\text{   }|P_.EG_.|\text{   }\| \times{|S_.G|}
\ar[d]^{{\text{pr}_1}} \ar[r]^{\phantom{1234567}ev\times{ev}}
&EG\times{G}  \ar[d]^{{\text{pr}_1}} \\
\|\text{   }|P_.EG_.|\text{   }\| \ar[r]^{\phantom{12}ev}
 & EG }
\end{displaymath}
Since $\|\text{   }|P_.EG_.|\text{   }\|$ and $EG$ are both
contractible, the evaluation map induces a homotopy equivalence on
the quotient.

\end{proof}

\newpage

\medbreak
\section{\bf{Lattice Gauge Theory, Parallel Transport Function}\label{six}}

In Lattice gauge theory in the sense of Phillips and Stone \cite{PS1} they construct for a given Lie group
$G$ and a simplicial complex $K$ a $G$-bundle with connection on $|K|$ associated to a set of $G$-valued
continuous functions defined over the faces of a simplex. These they call ``parallel transport functions''
since they are determined by parallel transport for the connection. In this section we shall introduce
similar ``compatible transition functions'' for $K$ replaced by a simplicial set $S$ and in the following
section we shall use these to construct a classifying map on the star complex $\bar{P_.}S_.$. First we
consider $G$-bundles over simplicial sets.

\begin{definition}
\label{6.2}
A bundle over $|S|$ is a sequence of bundles over
${\Delta}^p\times{\sigma}$ for all ${p}$, where $\sigma\in{S_p}$
and with commutative diagrams;
\begin{eqnarray*}
\xymatrix{
F_{{d}_j\sigma} \ar[d] \ar[r]^-{{\bar{\varepsilon}}^j}
&{F_{\sigma}}   \ar[d] \\
{\Delta}^{p-1}\times{{d}_j\sigma} \ar[r]^-{{\varepsilon}^j}
& {\Delta}^p\times{\sigma}
}
\end{eqnarray*}
and
\begin{eqnarray*}
\xymatrix{
F_{{s}_j\sigma} \ar[d] \ar[r]^-{{\bar{\eta}}^j}
&{F_{\sigma}}   \ar[d] \\
{\Delta}^{p+1}\times{{s}_j\sigma} \ar[r]^-{{\eta}^j}
& {\Delta}^p\times{\sigma}
}
\end{eqnarray*}
with the compatibility conditions:
\begin{eqnarray*}
{\bar{\varepsilon}}^j{\bar{\varepsilon}}^i= \left\{ \begin{array}{r@{\quad: \quad}l}{\bar{\varepsilon}}^i {\bar{\varepsilon}}^{j-1} & i<j  \\
{\bar{\varepsilon}}^{i+1}
{\bar{\varepsilon}}^j & i\geq{j}, \end{array} \right.
\end{eqnarray*}

\begin{eqnarray*}
{\bar{\eta}}^j{\bar{\eta}}^i= \left\{ \begin{array}{r@{\quad: \quad}l}{\bar{\eta}}^i {\bar{\eta}}^{j+1} & i\leq{j} \\
{\bar{\eta}}^{i-1}
{\bar{\eta}}^j & i>j, \end{array} \right.
\end{eqnarray*}
and
\begin{eqnarray*}
{\bar{\eta}}^j{\bar{\varepsilon}}^i= \left\{ \begin{array}{r@{\quad: \quad}l}{\bar{\varepsilon}}^i {\bar{\eta}}^{j-1} & i<j  \\
1 &i=j,i=j+1\\
{\bar{\varepsilon}}^{i-1}
{\bar{\eta}}^i & i>{j+1}. \end{array} \right.
\end{eqnarray*}
\vskip 0.3 cm Given a $G$-bundle $F\to{|S|}$, $G$ a Lie group, since
${\Delta}^p$ is contractible, we can choose a trivialization
${\varphi}_{\sigma}:F_{\sigma}\to{{\Delta}^p}\times{\sigma}\times{G}$
for a non-degenerate $\sigma\in{S_p}$. If $\sigma$ is degenerate,
that is, there exists $\tau$ such that $\sigma={s}_i\tau$, then the
trivialization of $\sigma$ is defined as pullback of the
trivialization of $\tau$, that is,
${\varphi}_{\sigma}={{\eta}^i}^*({\varphi}_{\tau})$.
\end{definition}

\begin{definition}
\label{6.3}
                    ( Admissible Trivializations )
A set of trivializations is called admissible, in case ${\varphi}_{\sigma}$ for $\sigma={s}_i\tau$ is given by
${\varphi}_{\sigma}={{s}^i}^*({\varphi}_{\tau})$.
\end{definition}

\begin{lemma}
\label{6.4}
Admissible trivializations always exist.

\end{lemma}

Now, let us construct the transition functions for a simplex $\sigma\in{S_p}$ before giving the following proposition:

\begin{definition}
\label{6.5}

Given a bundle and a set of trivializations, we get for each face ${\tau}$ of say
dim$\tau=q<p$ in $\sigma$, a transition function
$v_{\sigma,\tau}:{\Delta}^q\to{G}$.
E.g., if $\tau={d}_i\sigma$ then the transition function $v_{\sigma,{d}_i\sigma}:{\Delta}^{p-1}\to{G}$ is given by the diagram
\begin{displaymath}
\xymatrix{
{\Delta}^{p-1}\times{({d}_i\sigma)}\times{G} \ar[d]
\ar[r]^-{\Theta}
&{\Delta}^p\times{(\sigma)}\times{G}   \ar[d] \\
{\Delta}^{p-1}\times{{d}_i\sigma} \ar[r]^-{{\varepsilon}^i}
& {\Delta}^p\times{\sigma}
}
\end{displaymath}
where ${d}_i\sigma=\tau$ and $\Theta={\varphi}_{\sigma}\circ{{\bar{\varepsilon}}^i}\circ{{{\varphi}_{{d}_i\sigma}}^{-1}}$.
So
\begin{displaymath}
\{v_{\sigma,\tau}|\sigma\in{S_p}\text{   }\text{and} \text{   }\tau\text{   }\text{is a face of}\text{   }\sigma\}
\end{displaymath}
are {\bf the transition functions} for the bundle over $|S|$.
\end{definition}

{\bf Remark :}
The transitions functions are generalized lattice gauge fields. Classically Lattice gauge fields are defined only on 1-skeletons
but one can extend them to $p-1$
simplices for all $p$, given rise to transition functions on $\Delta^p$, as above.

We now list a number of propositions stating the properties of these. The proofs are straight forward. For details see Akyar \cite{B}.

\begin{proposition}
\label{6.6}

Given a bundle on a simplicial set and admissible trivializations, the transition function $v_{\sigma,\tau}$, where
$\tau$ is a face of $\sigma$, satisfies;

i) $\sigma$ is nondegenerate:
if $\gamma={d}_j\sigma$ and $\tau={d}_i\gamma$ then
\begin{displaymath}
v_{\sigma,\tau}=(v_{\sigma,\gamma}\circ{{\varepsilon}^i}).v_{\gamma,\tau}.
\end{displaymath}

This is called the cocycle condition.

ii) $\sigma$ is degenerate:
If $\sigma={s}_j{\sigma}^{\prime}$ and $\tau={d}_i\sigma$ then when $i<j$ for $\tau={s}_{j-1}{\tau}^{\prime}$ one gets
${\tau}^{\prime}={d}_i{ \sigma}^{\prime}$ and  when $i>j+1$ for $\tau={s}_j{\tau}^{\prime}$
one gets ${\tau}^{\prime}={d}_{i-1}{ \sigma}^{\prime}$. For the other cases, $i=j$ or $i=j+1$, $\tau={\sigma}^{\prime}$.
Then the transition functions satisfy:
\begin{displaymath}
v_{\sigma,\tau}= \left\{ \begin{array}{r@{\quad: \quad}l}v_{{\sigma}^{\prime},{\tau}^{\prime}}\circ{{\eta}^{j-1}} & i<j  \\
1 &i=j,i=j+1\\
v_{{\sigma}^{\prime},{\tau}^{\prime}}\circ{{\eta}^j}
& i>{j+1}. \end{array} \right.
\end{displaymath}

iii) If $\tau$ is a composition of face operators of $\sigma$, e.g., $\tau={\tilde{d}}^{p-(i-1)}\sigma$, $i=1,...,p$, where
${\tilde{d}}^{p-(i-1)}=d_i\circ{...}\circ{d_p}$ then
\begin{displaymath}
v_{\sigma,\tau}=(v_{\sigma,{\tilde{d}}^1\sigma}\circ{{({\varepsilon}^i)}^{p-i}}).
(v_{{\tilde{d}}^1\sigma,{\tilde{d}}^2\sigma}\circ{{({\varepsilon}^i)}^{p-i-1}})...
(v_{{\tilde{d}}^{p-(i+1)}\sigma,{\tilde{d}}^{p-i}\sigma}\circ{{\varepsilon}^i}).v_{{\tilde{d}}^{p-i}\sigma,\tau}.
\end{displaymath}
\end{proposition}

\begin{proposition}
\label{6.7}

Assume that we have a bundle over $|S|$. Then

1) There exists admissible trivializations
such that the transition function is given by
\begin{displaymath}
v_{\sigma,{d}_i\sigma}=1\text{    }\text{if}\text{   }i<p.
\end{displaymath}

2) For $\tau={\tilde{d}}^{p-(i-1)}\sigma$, $i=1,...,p$,
we get $v_{\sigma,\tau}$ as product of some transition functions:
\begin{displaymath}
v_{\sigma,\tau}=(v_{\sigma}\circ{{({\varepsilon}^i)}^{p-i}}).
(v_{{\tilde{d}}^1\sigma}\circ{{({\varepsilon}^i)}^{p-i-1}}).(v_{{\tilde{d}}^2\sigma}\circ{{({\varepsilon}^i)}^{p-i-2}})...
(v_{{\tilde{d}}^{p-(i+1)}\sigma}\circ{{({\varepsilon}^i)}^1}).(v_{{\tilde{d}}^{p-i}\sigma}).
\end{displaymath}

3) The transition functions $v_{\sigma,\tau}$ satisfy the compatibility
conditions:
\begin{displaymath}
v_{\sigma}\circ{{\varepsilon}^i}= \left\{ \begin{array}{r@{\quad: \quad}l}v_{{d}_i\sigma}
 & i<p-1  \\ v_{{d}_{p-1}\sigma}.{v_{{d}_p\sigma}}^{-1}
 & i={p-1} \end{array} \right.
\end{displaymath}

4) For a degenerate $\sigma$, we have
\begin{displaymath}
v_{\sigma}\circ{{\eta}^j}=v_{{s}_j\sigma}
\end{displaymath}    $\forall {j}$.
\end{proposition}

\begin{proposition}
\label{6.8}

 Given a bundle, one can find admissible trivializations such that the transition functions are determined by functions
$v_\sigma:{\Delta}^{p-1}\to{G}$ for $\sigma\in{S_p}$
nondegenerate.

\end{proposition}

\begin{proposition}
\label{6.9}

Suppose given a set of functions
\begin{displaymath}
v_{\sigma}:{\Delta}^{p-1}\to{G}
\end{displaymath}
for $\sigma\in{S_p}$ for all ${p}$, satisfying the compatibility
conditions
\begin{displaymath}
v_{\sigma}\circ{{\varepsilon}^i}= \left\{ \begin{array}{r@{\quad: \quad}l}v_{{d}_i\sigma}
 & i<p-1  \\ v_{{d}_{p-1}\sigma}.{v_{{d}_p\sigma}}^{-1}
 & i={p-1} \end{array} \right.
\end{displaymath}
and
\begin{displaymath}
v_{s_j\sigma}=v_{\sigma}\circ{\eta^j}.
\end{displaymath}

Then one can define for each $\sigma\in{S_p}$ and each lower dimensional face $\tau$ of $\sigma$, a function
$v_{\sigma,\tau}$ such that i) and ii)
in Proposition \ref{6.6} hold and such that
\begin{displaymath}
v_{\sigma,\tau}= \left\{ \begin{array}{r@{\quad: \quad}l}v_{\sigma} & i=p  \\
1 & i<p. \end{array} \right.
\end{displaymath}

\end{proposition}

\begin{proposition}
\label{6.10}

Given a set of transition functions $v_{\sigma,\tau}$ satisfying i) and ii) in Proposition \ref{6.6}, there is a bundle $F$ over $|S|$ and trivializations with
transition functions $v_{\sigma,\tau}$.

\end{proposition}

\begin{corollary}
\label{6.11}

Given a set of functions $v_{\sigma}$ satisfying the compatibility conditions in Proposition \ref{6.7}, one can construct a
bundle $F$ over $|S|$ and the trivializations with the transition functions $v_{\sigma,d_p\sigma}=v_{\sigma}$ and
$v_{\sigma,d_i\sigma}=1$ when $i<p$ and $v_{s_i\sigma}=v_{\sigma}\circ{\eta^i}$ for a degenerate $\sigma$.

\end{corollary}

\begin{definition}
\label{6.12}
A set of functions $\{v_{\sigma}\}_{\sigma\in{S}}$ as in Proposition \ref{6.9} are called a set of ``compatible transition functions''.
\end{definition}

We end this section by comparing these compatible transition
functions with the ``parallel transport functions'' (p.t.f.) of
Phillips and Stone \cite{PS1}. For $S=K^s$ these consist of a set of maps,
$V_\sigma:c_\sigma\to{G}$ for each $r$-simplex $\sigma$ of $K$,
$r\ge{1}$, $c_\sigma$ is the $(r-1)$-cube given by
$0\le{s_{a_1}}\le{1},...,0\le{s_{a_{r-1}}}\le{1}$, where
$\sigma=<a_0,...,a_r>\in{K}$ with the compatibility conditions

1. Cocycle condition
\begin{displaymath}
V_\sigma(s_{a_1},...,
s_{a_p}=1,
...,s_{a_{r-1}})=
V_{<a_0,...,a_p>}     (s_{a_1},...,s_{a_{p-1}}                                      ).
V_{<a_p,...,a_r>})(s_{a_{p+1}},...,s_{a_{r-1}}).
\end{displaymath}

2. Compatibility condition

\begin{displaymath}
V_\sigma(s_{a_1},..., s_{a_p}=0 ,...,s_{a_{r-1}})=
V_{<a_0,...,{\hat{a}}_p,...,a_{r-1}>}
(s_{a_1},...,{\hat{s}}_{a_p},...,s_{a_{r-1}}).
\end{displaymath}

Now, suppose we have compatible transition functions $\{v_\sigma\}$
for a principal $G$-bundle $E\to{|K|}$ with triangulated base. Then
for $\sigma=<a_0,...,a_r>$, the p.t.f. $V_\sigma:c_\sigma\to{G}$ is given
by the parallel transport $E_{a_0}\to{E_{a_r}}$ along paths determined as follows:

Let $\sigma=<a_0,...,a_r>\in{K^s}$ and $s=(s_{a_0},...,s_{a_{r-1}})\in{c_\sigma}$.

We pick $r-1$-points as $P_1,...,P_{r-1}$ so that $P_1$ is on the line segment from $a_0$ to $a_1$, that is,
\begin{displaymath}
P_1=(1-s_{a_1})a_0+s_{a_1}a_1=((1-s_{a_1},s_{a_1}),<a_0,a_1>)\in{|K|}.
\end{displaymath}
Similarly, $P_2$ is on the line segment from $P_1$ to $a_2$,
$P_2=(1-s_{a_2})P_1+s_{a_2}a_2$. Then
\begin{displaymath}
P_2=((1-s_{a_2})(1-s_{a_1}),(1-s_{a_2})s_{a_1},s_{a_2},<a_0,a_1,a_2>).
\end{displaymath}
By continuing in the same way, we get
\begin{displaymath}
P_{r-1}=(1-s_{a_{r-1}})P_{r-2}+s_{a_{r-1}}a_{r-1}.
\end{displaymath}
Let $\alpha$ be the piecewise linear path from $a_0$ through $P_1,...,P_{r-1}$ to $a_r$. In other words, $\alpha$
is determined uniquely up to parametrization by $r-1$ numbers $s_{a_1},...,s_{a_{r-1}}$.
For $P_{r-1}=(t,d_r\sigma)\in{\Delta^{r-1}}\times{K_{r-1}}$, $d_r\sigma=<a_1,...,a_{r-1}>$, the element\begin{displaymath}
V_{\sigma}(s_1,...,s_{r-1})=v_{\sigma}(t)\in{G}
\end{displaymath}
is to be interpreted as the parallel transport along $\alpha$.

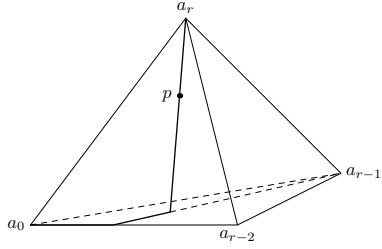
\begin{figure}[htbp]
  \includegraphics[width=5cm]{figs.2}\\
\end{figure}

\newpage

\medbreak
\section{\bf{The Classifying Map}\label{eight}}

{\bf {The construction of The Classifying Map}}

For a given set of compatible transition functions (c.t.f.) $\{v_\sigma\}$
satisfying Proposition \ref{6.9} we have seen in Proposition
\ref{6.10} that there is an associated $G$- bundle $F$ over $|S_.|$.
Recall that the composite map $\text{proj}\circ{L}:\|\text{
}|\bar{P}_.S|\text{   }\|\to{\|\text{   }|S_.|\text{   }\|}\to{|S|}$
is a homotopy equivalence, where $L=\Lambda\circ{f}$ is given as in Proposition 4.2. In this
section, we construct a classifying map for the bundle
$(\text{proj}\circ{L})^*F$ over $\|\text{
}|{\bar{P}}_.S_.|\text{ }\|$.

\begin{theorem}
  \label{7.4}
{\bf 1)} For given c.t.f.'s
$\{v_\sigma\}$, there is a canonical prismatic map \linebreak $m:{\bar{P}}_.S_.\to{P_.\text{BG}}$.

{\bf 2)} The induced map of geometric realizations
\begin{displaymath}
ev\circ{\|\text{   }|m|\text{   }\|}=\bar{m}:\|\text{
}|\bar{P}_.S|\text{ }\|\stackrel{\|\text{ }|m|\text{
 }\|}\to{\|\text{ }|P_.BG|\text{ }\|}\stackrel{ev}\to{BG}
\end{displaymath}
is a classifying map for the $G$-bundle $(\text{proj}\circ{L})^*F$ over $\|\text{   }|{\bar{P}}_.S_.|\text{   }\|$.
\end{theorem}

{\bf Proof :}

{\bf 1)} The map $m:{\bar{P}}_.S_.\to{P_.BG}$ is defined as
\begin{displaymath}
m(\sigma)=[(a_0,a_1,...,a_p)]
\end{displaymath}
where $\sigma\in{{\bar{P}}_pS_{q_0...q_p}}=S_{q+2p+1}$, $q=q_0+\dots+q_p$ and
$a_i:\Delta^p\times{\Delta^{q_0...q_i}}\to{G}$ are given below.
In the following, we use for convenience the interior coordinates $(t_1,\dots,t_p)$ of the standard simplex with 
barycentric coordinates $(t_0^\prime,\dots,t_p^\prime)$.
\begin{displaymath}
t_1=1-t_0^\prime,\text{   }t_2=1-t_0^\prime-t_1^\prime,\text{   }\dots\text{   }t_{r-1}=1-t_0^\prime-\dots-t_{r-1}^\prime,\text{   }t_r=t_r^\prime
\end{displaymath}
such that $0\leq{t_i}\leq{1}$, $i=1,\dots,p$, $1\geq{t_1}\geq{\dots}\geq{t_p}\geq{0}$ and $\sum_{i=0}^pt_i^\prime=1$, $t_i^\prime\leq{1}$, $i=0,\dots,p$.

In these terms the map $\Lambda$ from Section 3 is induced by the maps \linebreak
$\lambda_p:\Delta^p\times{\Delta^{q_0...q_p}}\to{\Delta^{q+2p+1}}$ given by

\begin{eqnarray*}
\lambda_p(t,s^0,0,...,0,s^p,0)&=&(s_1^0(1-t_1)+t_1,\dots,s_{q_0}^0(1-t_1)+t_1,t_1,t_1,\\
& &s_1^1(t_1-t_2)+t_2,\dots,s_{q_1}^1(t_1-t_2)+t_2,t_2,t_2,\\
& &...,\\
& &s_1^{p-1}(t_{p-1}-t_p)+t_p,\dots,s_{q_{p-1}}^{p-1}(t_{p-1}-t_p)+t_p,t_p,t_p,\\
& &s_1^pt_p,...,s_{q_p}^pt_p,0).
\end{eqnarray*}
For convenience, we drop $p$ in $\lambda_p(t)(s)$ and write
$\lambda(t)(s)$. Next, let  \linebreak $\rho^{(i)}:\Delta^{q+2p+1}\to{\Delta^{q_0+\dots+q_{i-1}+2i-1}}$ be the 
degeneracy map for $i=1,\dots,p$ defined by 
\begin{displaymath}
\rho^{(i)}:=\eta^{q_0+\dots+q_{i-1}+2i-1}\circ{\dots}\circ{\eta^{q+2p}}
\end{displaymath}
deleting the last $q_i+\dots+q_p+2(p-i+1)$ coordinates. 
So e.g.
\begin{eqnarray*}
\rho^{(p)}\lambda(t)(s)&=&(s_1^0(1-t_1)+t_1,\dots,s_{q_0}^0(1-t_1)+t_1,t_1,t_1,\\
& &s_1^1(t_1-t_2)+t_2,\dots,s_{q_1}^1(t_1-t_2)+t_2,t_2,t_2,\\
& &...,\\
& &s_1^{p-1}(t_{p-1}-t_p)+t_p,\dots,s_{q_{p-1}}^{p-1}(t_{p-1}-t_p)+t_p,t_p),
\end{eqnarray*}
where $\rho^{(p)}:=\eta^{q-q_p+2p-1}\circ{...}\circ{\eta^{q+2p}}$ is
deleting the last $q_p+2$ coordinates.
With this notation, the maps $a_i:\Delta^p\times{\Delta^{q_0...q_i}}\to{G}$
defining the classifying map $m(\sigma)$ are given by
\begin{eqnarray*}
a_p(t,s^0,0,...,s^p,0)&=&1\\
a_{p-1}(t,s^0,0,...,s^{p-1},0)&=&v_{\sigma,d_{(p)}\sigma}(\rho^{(p)} (\lambda(t)(s))){}^{-1},\\
a_{p-2}(t,s^0,0,...,s^{p-2},0)&=&v_{\sigma,{\tilde{d}}^{(2)}\sigma}(\rho^{(p-1)} (\lambda(t)(s))){}^{-1}\\
&.&\\
&.&\\
&.&\\
a_1(t,s^0,0,s^1,0)&=&v_{\sigma,d_{(2)...(p)}\sigma}(\rho^{(2)} (\lambda(t)(s))){}^{-1}\\
a_0(t,s^0,0)&=&v_{\sigma,d_{(1)...(p)}\sigma}(\rho^{(1)} (\lambda(t)(s))){}^{-1}.
\end{eqnarray*}
Here the boundary operators used above are given as follows:
\begin{displaymath}
d_{(p)}:S_{q+2p+1}\to{S_{q+2p-q_p-1}}
\end{displaymath}
is defined by $d_{(p)}:=d_{q+2p-q_p}\circ{...}\circ{d_{q+2p+1}}$,
deleting $q_p+2$ elements. On the other hand, in the formula
${\tilde{d}}^{(1)}={\hat{d}}_{(p)}=d_{(p)}$. Let's denote
\begin{displaymath}
{\tilde{d}}^{(p-i)}={\hat{d}}_{(i+1)}\circ{...}\circ{{\hat{d}}_{(p)}}
\end{displaymath}
$i=0,...,p-1$, which deletes the elements
$(q_0+...+q_i+2i-1,...,q+2p+1)$. It deletes \linebreak
$q_{i+1}+...+q_p+2(p-i)=q-(q_0+...+q_i)+2(p-i)$ elements. Here
\begin{displaymath}
{\hat{d}}_{(i)}:S_{q+2i+1-\sum_0^{p-i-1}q_{p-j}}\to{S_{q+2i-1-\sum_0^{p-i}q_{p-j}}}
\end{displaymath}
$i=1,...,p$. By using the equivalence relations on $m$ we can see
that $m(d_{(i)}\sigma)$ is independent of $s^i$ for all $j$
different from $i$. Take $t_0^\prime=0$ then
$v_{\sigma,{\tilde{d}}^{(p)}\sigma}(1,...,...,1,1,1)$ does not
depend on $s^0$ where $j=1\neq{0}=i$.

{\bf 2)} For given c.t.f.'s $v_{\sigma}$,
we now have the map of realizations \linebreak
$\|\text{   }|m|\text{   }\|:\|\text{   }|{\bar{P}}_.S_.|\text{   }\|\to{\|\text{   }|P_.BG|\text{   }\|}$ given by
\begin{eqnarray*}
\|\text{   }|m|\text{   }\|(t,s,\sigma)=(t,s,[(a_0,\dots,a_p)]).
\end{eqnarray*}

The associated bundle map is given as follows:

We have a bundle $F$ on $|S|$ by Proposition \ref{6.10} and
$|{\bar{P}}_pS_.|\to{|S_.|}$ is an epimorphism, so by pulling back
we get a bundle $\bar{F}\to{|{\bar{P}}_pS_.|}$, i.e.,
\begin{displaymath}
\xymatrix{
\bar{F} \ar[d] \ar[r]
&F   \ar[d] \\
|{\bar{P}}_.S_.| \ar[r]
& |S_.|
}
\end{displaymath}
Transition functions used to define the classifying
map $\tilde{m}$ are taken from the bundle $F\to{|S|}$.  Let's take
$\sigma\in{S_{q+2p+1}}$ and there is a fibre at
$(\lambda(t,s^0,0\dots,s^p,0),\sigma)$, by using the trivialization
$\varphi_\sigma:F_\sigma\to{{\Delta}^{q+2p+1}\times{\sigma}\times{G}}$
and the projection on the last factor, we get $F_\sigma\to{G}$.
Let's denote this composition by ${\bar{\varphi}}_\sigma(\tilde{f})$
where $\tilde{f}:=(\lambda(t,s^0,0,\dots,s^p,0),\sigma)$,
$\tilde{f}_\sigma\in{F_{(\lambda(t,s^0,0,\dots,s^p,0),\sigma)}}$,
$\sigma\in{S_{q+2p+1}}$. On the other hand
\begin{displaymath}
{\varphi}_{d_{(p)}\sigma}:F_{d_{(p)}\sigma}\to{{\Delta}^{q+2p-q_p-1}}\times{d_{(p)}\sigma}\times{G}
\end{displaymath}
gives us
\begin{displaymath}
{\bar{\varphi}}_{d_{(p)}\sigma}:F_{d_{(p)}\sigma}\to{G}.
\end{displaymath}

By the definition,
\begin{displaymath}
{\bar{\varphi}}_\sigma({\bar{d}}^{(p)}\tilde{f}_\sigma):=v_{\sigma,d_{(p)}\sigma}(\rho^{(p)}\lambda(t,s^0,0,\dots,s^p,0)).{\bar{\varphi}}_{d_{(p)}\sigma}{(\tilde{f})}_{d_{(p)\sigma}},
\end{displaymath}
where the compatible transition function is
\begin{displaymath}
v_{\sigma,d_{(p)}\sigma}:{\Delta}^{q+2p-q_p-1}\to{G}.
\end{displaymath}

The last component in $\|\text{   }|m|\text{   }\|(t,s,{\tilde{f}}_\sigma)$ is defined via the trivialization $\varphi_\sigma(\tilde{f})$ which is ${\bar{\varphi}}_\sigma(\tilde{f})$. By using the compatible transition function $v_{\sigma,d_{(p)}\sigma}$ we find the $p$-th component as
\begin{displaymath}
{v_{\sigma,d_{(p)}\sigma}(\rho^{(p)} \lambda(t,s^0,0,\dots,s^p,0))}^{-1}.{\bar{\varphi}}_\sigma(\tilde{f}).
\end{displaymath}
We can apply the same method several times to get the other coordinates in $\|\text{   }|m|\text{   }\|(t,s,{\tilde{f}}_\sigma)$.

By the definition $PEG/SG=PBG$, $PEG=\|N\bar{G}\|$ and $\gamma:N\bar{G}\to{NG}$ we can identify $PBG=\|NG\|$. Then the required map $\bar{m}$ is
\begin{displaymath}
\bar{m}(t,s,\sigma)=[(a_0,...,a_p)].
\end{displaymath}

$\hfill \Box$

In particular for a simplicial complex $K$ we get the following (c. f. \cite{PS1})

\begin{corollary}(Phillips-Stone)
{\bf 1)} A set of compatible transition functions $\{v_\sigma\}$ for $K$ a simplicial complex there is a natural prismatic map
\begin{displaymath}
P_.\St(K^s)\to{PBG}.
\end{displaymath}

{\bf 2)} The induced map on geometric realization gives a classifying map for the bundle $F$ pulled back to $|\St(K)|\subseteq{|K|\times{|K|}}$.

\end{corollary}

{\bf Proof:}
In the second part of Theorem \ref{4.4}, we have showed that \linebreak
$\bar{p}:\bar{P}K^s\to{P\text{St}K^s}$ is an isomorphism. On the other hand in the
previous proposition, we have defined the classifying map $m$. This is also valid when $S=K^s$. So the p.t.f. $v_\sigma$ will determine a natural map
\begin{displaymath}
m:P\text{St}K^s\to{PBG}.
\end{displaymath}

Furthermore $\pi_1:P\text{St}(K^s)\to{K}$ is a homotopy equivalence.

{\bf Remark:}
The point of the corollary is that there is a connection in the prismatic universal bundle in the simplicial sense (see \cite{DLj}) which thus pulls 
back to a connection in the bundle over the star complex. We shall return to this elsewhere.

\end{document}